\documentclass{article}

\usepackage{standalone}
\usepackage{import}
\usepackage[utf8]{inputenc}
\usepackage{amssymb}
\usepackage{amsmath}
\usepackage{mathrsfs}
\usepackage{amsthm}
\usepackage[shortlabels]{enumitem}
\usepackage{xparse}
\usepackage{tikz-cd}
\usepackage{mathtools}
\usepackage{bbm}
\usepackage{enumitem}
\usepackage[a4paper, total={6in, 8in}]{geometry}
\usepackage{hyperref}

\title{Ekedahl-Oort Types and Newton Polygons of Abelian Covers of $\mathbf{P}^1$ Branched at Three Points}
\author{Darren Schmidt}
\date{}

\newtheorem{theorem}{Theorem}[section]
\newtheorem{corollary}[theorem]{Corollary}
\newtheorem{lemma}[theorem]{Lemma}
\newtheorem{proposition}[theorem]{Proposition}
\newtheorem{expectation}[theorem]{Expectation}
\newtheorem{conjecture}[theorem]{Conjecture}

\theoremstyle{remark}
\newtheorem{remark}[theorem]{Remark}

\theoremstyle{definition}
\newtheorem{definition}[theorem]{Definition}
\newtheorem{construction}[theorem]{Construction}
\newtheorem{example}[theorem]{Example}

\begin{document}

\maketitle

\numberwithin{equation}{section}

\begin{abstract}
    In this paper, we study the Newton polygons and Ekedahl-Oort types of reductions of abelian covers of the projective line branched at three points modulo a prime. We study the natural density of primes where these covers give supersingular and superspecial curves and show they appear much more often than expected. We also show that unlikely Newton polygons and Ekedahl-Oort types in the moduli space of curves appear frequently. Finally, we prove a theorem that provides evidence of Oort's Conjecture about Newton polygons in certain cases and gives new constructions of supersingular curves.
\end{abstract}
\let\thefootnote\relax\footnotetext{We thank Jeremy Booher for helpful comments.}

\section{Introduction}
\par \hspace{\parindent}
Let $p$ be a prime and $k$ be an algebraically closed field of characteristic $p$. If $X$ is a smooth projective curve of genus $g$ over $k$, then the Jacobian of $X$, Jac$(X)$, is an abelian variety over $k$. Abelian varieties in positive characteristic can be studied through their invariants, such as their Newton polygon \cite[Section 4.2]{pries_torelli} or their Ekedahl-Oort type \cite[Section 4.4]{pries_torelli}. If $g$ is 1,2, or 3 then we know that every symmetric Newton polygon  of height $2g$ with endpoints $(0,0)$ and $(2g,g)$ and integer breakpoints occurs for the Jacobian of a smooth projective curve of genus $g$. In general, we do not know exactly which symmetric Newton polygons or Ekedahl-Oort types occur for Jacobians of smooth projective curves.

\par In this paper, we study which Newton polygons and Ekedahl-Oort types occur for Jac($X$) when $X$ is an abelian cover of the projective line $\mathbf{P}^1$ branched at exactly three points. We have that $X$ is the quotient of a Fermat curve, and therefore Jac$(X)$ has complex multiplication. This makes the actions of Frobenius and Verschiebung on the cohomology of the curve easier to analyze, which lets us determine the Newton polygon and Ekedahl-Oort type of $X$. We will see that they are often ``unlikely'' in a sense we now make precise.

\par Let $\mathcal{M}_g$ be the moduli space of smooth projective curves of genus $g$ in characteristic $p$ and $\mathcal{A}_g$ be the moduli space of principally polarized abelian varieties of dimension $g$ in characteristic $p$. We know that if $g \geq 2$, $\mathcal{M}_g$ has dimension $3g-3$ and $\mathcal{A}_g$ has dimension $g(g+1)/2$ \cite[Section 5.3]{pries_newtonpolygons}. 
\begin{definition}
\label{unlikely}
    If $\xi$ is a Newton polygon (or Ekedahl-Oort type) in dimension $g$, let $\mathcal{M}_g[\xi]$ be the moduli space of curves of genus $g$ in characteristic $p$ with Newton polygon (respectively, Ekedahl-Oort type) $\xi$ and let $\mathcal{A}_g[\xi]$ be the moduli space of principally polarized abelian varieties of dimension $g$ with Newton polygon (respectively, Ekedeahl-Oort type) $\xi$. We say that $\xi$ is an \textbf{unlikely} Newton polygon (Ekedahl-Oort Type) if the codimension of $\mathcal{A}_g[\xi]$ in $\mathcal{A}_g$ exceeds the dimension of $\mathcal{M}_g$. We say $\xi$ is \textbf{realized} in characterstic $p$ if $\mathcal{M}_g[\xi]$ is non-empty.
\end{definition}

\par If the Torelli locus and the Newton polygon strata (or Ekedahl-Oort strata) were dimensionally transverse, unlikely Newton polygons (or Ekedahl-Oort types) would not occur.  This na\"{i}ve heuristic is known to be incorrect. (See for example the existence of supersingular curves of any genus in characteristic two \cite{MR1310953}, Pieper's analysis of the hyperelliptic locus in genus three \cite{pieper}, and the examples in \cite{booher_supersingular_2025}.) We provide additional illustrations that this heuristic is far from correct using abelian covers of the projective line $\mathbf{P}^1$ branched at exactly three points. Some of our results are computational --- they then inspired some theoretical results --- see \ref{1.1} and \ref{1.2}.

\par Supersingular and superspecial curves are commonly studied types of curves, and we pay particular attention to them. If $E$ is an elliptic curve over $k$, we say that $E$ is supersingular if $E(k)[p] = 0$. We call $X$ supersingular if Jac$(X)$ is isogenous to a product of $g$ supersingular elliptic curves. In a similar vein, we call $X$ superspecial if Jac$(X)$ is isomorphic to a product of $g$ supersingular elliptic curves. We can determine if a curve is supersingular by examining its Newton polygon, and similarly we can deterime if a curve is superspecial by examining its Ekedahl-Oort type. If the genus is large enough, then supersingular and superspecial curves have an unlikely Newton polygon and an unlikely Ekedahl-Oort type, respectively.

\begin{remark}
    We note that if $\xi$ is a supersingular Newton polygon of dimension $g$ then $\mathcal{A}_g[\xi]$ has dimension $\lfloor g^2/4 \rfloor$ \cite[Section 2.4]{pries_newtonpolygons}. Therefore, if $g \geq 9$, then a supersingular curve of genus $g$, if it exists, has an unlikely Newton polygon. Also, if $\xi$ is a superspecial Ekedahl-Oort type of dimension $g$, then $\mathcal{A}_g[\xi]$ has dimension $0$ \cite[Page 1]{oort_denominators}. We then see that if $g \geq 4$, then a superspecial curve has an unlikely Ekedahl-Oort type. Supersingular and superspecial curves arising from abelian covers of $\mathbf{P}^1$ branched at three points have been studied before. In \cite{ekedahl_supersingular_1987}, Ekedahl determined for which primes a superspecial curve of this type occurs when $g = 4,5$. 
\end{remark}

\subsection{Computational Results} \label{1.1}

\par \hspace{\parindent} In \cite{booher_supersingular_2025}, for a fixed genus $g$, Booher and Pries determined the natural density of primes $p$ such that a supersingular curve that is an abelian cover of $\mathbf{P}^1$ branched at three points and defined over $\overline{\mathbf{F}_p}$ occurs for $5 \leq g \leq 10$. In particular, they wrote code in SageMath that can compute the Newton polygon when the curve is a cyclic cover. In the case where the curve is a non-cyclic cover, they decomposed the Jacobian of the curve by hand into a product of cyclic cases to determine the Newton polygon. In this paper, we extend their work by creating an algorithm in SageMath that can handle the non-cyclic case. In doing so, for genus up to 50 we are able to utilize SageMath to compute the natural density of primes where a supersingular curve occurs. We also write a program to compute the Ekedahl-Oort type and determine the natural density of primes such that there exist a cover that gives a superspecial curve, which we use to fix an error in \cite{ekedahl_supersingular_1987} for genus 5 --- see Proposition \ref{eke}. Furthermore, we also use this program to find additional curves with unlikely Newton polygons or Ekedahl-Oort types and compute the natural density of primes where they occur.

\par From our computations, we determined that supersingular curves, superspecial curves, unlikely Newton polygons, and unlikely Ekedahl-Oort types occur much more frequently than expected for such a specific family of curves. Let $\delta_{ss}(g)$, $\delta_{ssp}(g)$, and $\delta_{nu}(g)$ be the densities of primes such that there exists an abelian cover of $\mathbf{P}^1$ branched at three points of genus $g$ that is supersingular, superspecial, and that has an unlikely Newton polygon, respectively. Then, if $9 \leq g \leq 50$, 
\[
\begin{array}{ccc}
    \delta_{ss}(g) > 0.7, &  \delta_{ssp}(g) > 0.2, & \delta_{nu}(g) > 0.875   .
\end{array}
\]
The data we collected led us to make the following conjecture: 

\begin{conjecture} \label{unlikely-density}
    Let $g \geq 18$, and let $m$ be the least common multiple of the exponents of all of the Galois groups of abelian covers of $\mathbf{P}^1$ branched at three points of genus $g$. If $p$ is a prime such that $p \not \equiv 1$ modulo $m$, then there exist abelian covers $X_1$ and $X_2$ of $\mathbf{P}^1$ branched at three points defined over $\overline{\mathbf{F}_p}$ such that $X_1$ has an unlikely Newton polygon and $X_2$ has an unlikely Ekedahl-Oort type.
\end{conjecture}

\subsection{Theoretical Results} \label{1.2}

\par \hspace{\parindent} After noticing a pattern in the results of these computations, we are also able to prove several results about the following conjecture and expectation of Oort:

\begin{conjecture}[\protect{\cite[Conjecture 8.5.7]{oort_denominators}}]
    \label{oort_conjecture1}
    If $p$ is a prime and $\xi_1$ and $\xi_2$ are Newton polygons that are realized in characteristic $p$, then $\xi_1 \oplus \xi_2$ is realized in characteristic $p$.
\end{conjecture}

\begin{expectation}[\protect{\cite[Expectation 8.5.4]{oort_denominators}}]
    \label{expectation1}
    Given a symmetric Newton polygon $\xi$ of height $2g$ such that $\xi$ is unlikely and has slopes with denominators (say, greater than 11), then $\xi$ is not realized in characteristic $p$ for all primes $p$.
\end{expectation}

\begin{theorem}
\label{slopes}
Let $\ell > 3$ be a prime and $n \geq 1$ be coprime to $\ell$. Define $g := (\ell-1)/2$ and suppose $\alpha$ is the number of quadratic residues modulo $\ell$ contained in the intervals $(0,\ell/4)$, $(\ell/3,\ell/2)$, and $(2\ell/3, 3\ell/4)$. If $p$ is a prime with order $g$ in $\left (\mathbf{Z}/\ell\mathbf{Z} \right)^\times$ such that $gcd(p,n\ell) = 1$, and the order of $p$ in $\left (\mathbf{Z}/n\mathbf{Z} \right)^\times$ is coprime to $g$, then there exists a cyclic cover $\phi: X \to \mathbf{P}^1$ branched at three points defined over $\overline{\mathbf{F}_p}$ such that $X$ has genus $ng$ and the only slopes of the the Newton polygon of $\textrm{Jac}(X)$ are $\alpha/g$ and $(g-\alpha)/g$.
\end{theorem}

\par Theorem \ref{slopes} is a generalization of \cite[Corollary 3.10]{pries_cyclic}, which focuses on the case when $g$ is a Sophie Germain prime. For each Sophie Germain prime, \cite[Corollary 3.10]{pries_cyclic} constructs a finite number of supersingular curves; however, Theorem \ref{slopes} gives us infinitely many examples of supersingular curves without having to assume that there are infinitely many Sophie Germain primes. It also gives us examples of Newton polygons that satisfy Conjecture \ref{oort_conjecture1}, and it also can be used to construct counterexamples to Expectation \ref{expectation1}.

\par The following corollary is immediate and gives infinitely many Newton polygons that satisfy Conjecture \ref{oort_conjecture1}

\begin{corollary}\label{conjecture-corollary}
Let $\ell > 3$ be a prime. If $p > 2$ is a prime with $p \ne \ell$ such that $p$ is congruent to a non-zero quadratic residue modulo $\ell$, then there is a Newton polygon $\xi$ such that $\xi$ is realized in characteristic $p$ and $\xi \oplus \xi$ is realized in characteristic $p$.
\end{corollary}

The next two corollaries show how to use Theorem \ref{slopes} to construct supersingular curves.

\begin{corollary}\label{germain1}
    Let $\ell>3$ be a prime with $\ell \equiv 1$ modulo $4$ and $n \geq 1$ be coprime to $\ell$. Define $g := \frac{\ell-1}{2}$. If $p$ is a prime with order $g$ in $\left (\mathbf{Z}/\ell\mathbf{Z} \right)^\times$ such that $gcd(p,n\ell) = 1$, and the order of $p$ in $(\mathbf{Z}/n\mathbf{Z})^\times$ is coprime to $g$, then there exists a supersingular curve of genus $ng$ defined over $\overline{\mathbf{F}_p}$.
\end{corollary}

\begin{corollary}
\label{germain2}
    Let $\ell > 3$ be a prime and $n \geq 1$ be coprime to $\ell$. Define $g := \frac{\ell-1}{2}$. If $p$ is a prime with order $2g$ in $\left (\mathbf{Z}/\ell\mathbf{Z} \right)^\times$ such that $gcd(p,n\ell) = 1$, and the order of $p$ in $(\mathbf{Z}/n\mathbf{Z})^\times$ is coprime to $g$, then there exists a supersingular curve of genus $ng$ defined over $\overline{\mathbf{F}_p}$.
\end{corollary}

\par In particular, if $g>2$ is a Sophie Germain prime, then $\ell= 2g+1$ is $3$ modulo $4$, so Corollary \ref{germain2} with $n = 1$ gives the same supersingular curves as \cite[Corollary 3.10]{pries_cyclic}. We are also able to use Theorem \ref{slopes} to give explicit counterexamples to Expectation \ref{expectation1}.

\begin{proposition}
\label{denominators}
    If $g > 363$ is a Sophie Germain prime, then for every $n \geq 1$ such that $gcd(n,2g+1)=1$, there exists a curve of genus $ng$ whose Newton polygon is unlikely and has slopes with denominator $g$. 
\end{proposition}

\par We note that \cite[Corollary 3.10]{pries_cyclic} gives the same Newton polygons as Proposition \ref{denominators} for $n=1$, but Proposition \ref{denominators} shows we have infinitely many unlikely Newton polygons with large denominators even though we don't know if there are infinitely many Sophie Germain primes.

\par Finally, we are able to use Theorem \ref{slopes} to prove the following results about the natural density $\delta_{ss}(g)$:
\begin{proposition} \label{limsup}
    We have that $\limsup_{g \to \infty} \delta_{ss}(g) > 0.9999$.
\end{proposition}

Furthermore, if we assume Dickson's Conjecture --- stated in Conjecture \ref{conjecture:dickson} --- then we can compute the value of $\limsup_{g \to \infty} \delta_{ss}(g)$. 

\begin{theorem}
\label{dickson}
    If Dickson's Conjecture is true, then $\limsup_{g \to \infty} \delta_{ss}(g) = 1$.
\end{theorem}

\par Section 2 of this paper gives a background on abelian covers of $\mathbf{P}^1$ branched at three points and explains how to utilize the fact that the Jacobian has complex multiplication to compute Newton polygons and Ekedahl-Oort types. Section 3 gives computational results on the natural density of primes that give supersingular and superspecial curves for each genus, and also provides data on how often unlikely Newton polygons and Ekedahl-Oort types occur. Section 4 gives the proof of Theorem \ref{slopes} and its corollaries by utilizing the Shimura-Taniyama formula to analyze the orbits of a specific cover, and Section 5 utilizes Theorem \ref{slopes} to prove results about natural densities and Newton polygons which has slopes with large denominators.
\section{Abelian Covers}

\par \hspace{\parindent} Let $X$ be a smooth projective connected curve over $k$, where char$(k) = p$. Let $\phi: X \to \mathbf{P}^1$ be a Galois cover branched at exactly $B := \{0,1,\infty\}$. Let $g$ be the genus of $X$. Assume that $G$ is abelian and that $p$ does not divide the order of $G$.

\subsection{Ramification and Inertia Types}
\par \hspace{\parindent} If $b \in B$, let $I_b$ be the inertia group above $b$ and define $c_b = |I_b|$. As $p$ does not divide the order of $G$, $\phi$ is tamely ramified; therefore, $I_b$ is cyclic. We assume without loss of generality that $c_0 \geq c_1 \geq c_\infty$.
\begin{definition}
    The \textit{ramification type} of the cover $\phi$ is the tuple $z = [c_0,c_1,c_\infty, s]$ where $s$ is the index of $I_0$ in $G$.
\end{definition}
\par As $p$ does not divide the order of $G$, $\phi$ is tamely ramified. Therefore, the ramification type determines the genus of $X$ by the Riemann-Hurwitz formula as in \cite[Lemma 2.3]{booher_supersingular_2025}.
\begin{definition}
    An \textit{inertia type} associated to a ramification type $z = [c_0,c_1,c_\infty, s]$ is the tuple $a = (a_0, a_1, a_\infty)$ of elements of $G$ such that 
    \begin{enumerate}[(1)]
        \item $|a_b| = c_b$ for $b \in B$
        \item $a_0+a_1+a_\infty = 0$ in $G$
        \item $a_0$, $a_1$, and $a_\infty$ generate $G$
    \end{enumerate}
\end{definition}
\par For a given ramification type $z$, we can relabel the branch points if their inertia groups have the same order. Furthermore, we can modify the inertia type by applying automorphisms of $G$. If two inertia types are in the same orbit under these two actions, we say they are equivalent.

\subsection{Cyclic Covers}
\par \hspace{\parindent} Assume that $\phi$ is a cyclic cover of degree $m$. Then, $G \simeq \mathbf{Z}/m\mathbf{Z}$. The abelian variety Jac$(X)$ has complex multiplication by $\prod_d  \mathbf{Q}(\zeta_d)$ where the product is taken over all integers $d$ such that $1 < d  \ | \ m$ and $d \nmid a_b$.

\par We also know from \cite{pries_newtonpolygons} that if $1 \leq j < m$ then the dimension of the $\zeta_m^j$-eigenspace of $H^0(X,\Omega_X^1)$, the space of regular differentials on $X$, is either $0$ or $1$.

\begin{definition}
    Let $f: \{1, 2, \ldots, m-1 \} \to \{0,1\}$ be the function that sends $j$ to the dimension of the $\zeta_m^j$-eigenspace of $H^0(X,\Omega_X^1)$. We call $f$ the \textit{signature type} of $\phi$.
\end{definition}

\par If $q$ is a rational number, let $\langle q \rangle$ denote the fractional part of $q$. If we have a cyclic cover with ramification type $[c_0,c_1,c_\infty]$ and inertia type $(a_0,a_1,a_\infty)$, then the Chevalley-Weil formula \cite[Theorem 2.10]{frediani} gives us that \[f(j) = -1 + \sum_{b \in B}\left \langle \frac{-ja_b}{m}\right \rangle.\]

\subsection{Newton Polygons and Ekedahl-Oort Types}
\begin{definition}[\protect{\cite[8.4.1]{oort_denominators}}]
    Let $g$ be a non-negative integer. A \textbf{symmetric Newton polygon of height $\mathbf{2g}$} is a polygon in $\mathbf{Q} \times \mathbf{Q}$ such that
    \begin{itemize}
        \item The polygon starts at $(0,0)$ and ends at $(2g,g)$,
        \item The polygon is lower convex,
        \item The breakpoints are in $\mathbf{Z} \times \mathbf{Z}$,
        \item The slopes are of the form $\lambda \in \mathbf{Q}$ with $0 \leq \lambda \leq 1$,
        \item If a slope $\lambda$ appears, then the slope $1-\lambda$ appears with the same multiplicity.
    \end{itemize}
\end{definition}

\par We can view the Newton polygon as being a multiset of rational numbers representing the slopes with their multiplicities. A curve is ordinary if the only slopes are $0$ and $1$, and a curve is supersingular if the only slopes are $1/2$.

\begin{definition}
    The \textbf{Ekedahl-Oort type} of $X$ is the isomorphism class of the group scheme Jac$(X)[p]$.
\end{definition}

\par By \cite{oda}, the Dieudonn\'e module of Jac$(X)[p]$ is isomorphic to the first de Rham cohomology of Jac$(X)$, so the isomorphism class of Jac$(X)[p]$ is uniquely determined $H^1_{dR}(X)$. Let $F$ and $V$ be the Frobenius operator and Cartier operator on $H^1_{dR}(X)$, respectively. As an example, if $X$ is a supersingular elliptic curve over $k$, then $H^1_{dR}(X)$ is a 2 dimensional vector space over $k$ with basis $\{e_1,e_2\}$ such that $F(e_1) = V(e_1) = e_2$ and $F(e_2) = V(e_2) = 0$.

\subsection{Non-cyclic Covers}
\par \hspace{\parindent} Assume that $\phi$ is a degree $m$ cover that is non-cyclic, and let $a$ be the inertia type of $\phi$. We know that $a_0$, $a_1$, and $a_\infty$ generate $G$ and that $a_0 + a_1 + a_\infty = 0$ in $G$. Therefore, $G$ is generated by $a_0$ and $a_1$. Since $G$ is not cyclic, it cannot be generated by just one element, so we must have that $G$ is of the form $\mathbf{Z}/c\mathbf{Z} \times \mathbf{Z}/d \mathbf{Z}$ such that $d | c$ and $cd = m$.

\par Let $p$ be a prime that does not divide $m$, the order of $G$. Let $\mathcal{T}_G := \textrm{Hom}(G,\mathbf{C}^\times)$ be the Pontryagin dual of $G$. There is a non-canonical isomorphism between $\mathcal{T}_G$ and $G$, so we can identify the elements of the Pontryagin dual with the elements of $G$. Also, if we quotient $G$ by the kernel of an element of $\mathcal{T}_G$, then the result is a finite subgroup of $\mathbf{C}^\times$, and is therefore a cyclic group. 

\par Denote $\mathcal{O}_G$ as the set of orbits of the Frobenius action on $\mathcal{T}_G$.

\begin{proposition}
[\protect{\cite[Section 2.2.2]{lin_abelian_2024}}]
    If $M = \mathbf{D}(\textrm{Jac}(X))$ is the Dieudonn\'{e} module of Jac$(X)$, then we have a decomposition up to isogeny given by $M \simeq \bigoplus_{\mathcal{O} \in \mathcal{O}_G} M_{\mathcal{O}}$, where $M_{\mathcal{O}}$ is the sum of the $\tau$-isotypic components of $M$ for $\tau \in \mathcal{O}$. Let $NP(M)$ be the Newton polygon of $M$. Since the Newton polygon of two Dieudonn\'{e} modules are equal if and only if the modules are isogenous, we have that $NP(M) = \bigoplus_{\mathcal{O} \in \mathcal{O}_G} NP(M_\mathcal{O})$, where the direct sum of Newton polygons is constructed by combining the slopes of each summand.
\end{proposition}

\par Therefore, we can reduce to the cyclic case. For the Ekedahl-Oort type we have a similar proposition:
\begin{proposition}
    [\protect{\cite[Section 2.3.1]{lin_abelian_2024}}] As mod-$p$ Dieudonn\'{e} modules, $H^1_{dR}(X) \simeq \bigoplus_{\mathcal{O} \in \mathcal{O}_G} H^1_{dR}(X)_{\mathcal{O}}$, where $H^1_{dR}(X)_{\mathcal{O}}$ is the sum of the $\tau$-isotypic components of $H^1_{dR}(X)$ for $\tau \in \mathcal{O}$.
\end{proposition}

\subsection{Computing Newton Polygons and Ekedahl-Oort Types}
\par If $G \simeq \mathbf{Z}/m\mathbf{Z}$ and $p$ is a prime not dividing $m$, let $\mathcal{O}_G$ be the set of orbits of the action of multiplication by $p$ on the set $\{j \in G \mid j \neq 0\}$. If $d_j$ is the order of $j$ in $G$, define \[S_1 := \{1 \leq j < m \mid d_j \textrm{ does not divide } a_b \textrm{ for any } b \in B \textrm{ and } f(j)=1\}\]
\par We note that if $\mathcal{O} \in \mathcal{O}_G$, then the order of any element of $\mathcal{O}$ is the same. Let $d_\mathcal{O}$ be the order of any element of $\mathcal{O}$ in $G$.
\par Let $\mathcal{O} \in \mathcal{O}_G$, the Shimura-Taniyama formula  states that 

\begin{theorem}[Shimura-Taniyama Formula]
    \cite[Section 5]{SB_1968-1969__11__95_0} The only slope of the Newton polygon of Jac$(X)[p^\infty]_\mathcal{O}$ is $\lambda = \#(\mathcal{O} \cap S_1)/\#\mathcal{O}$.
\end{theorem}
\par To compute the slopes of the Newton polygon of Jac$(X)$, we compute the slope of the Newton polygon of Jac$(X)[p^\infty]_\mathcal{O}$ for each $\mathcal{O} \in \mathcal{O}_G$ using the Shimura-Taniyama formula. Each of the slopes appears as a slope of the Newton polygon of Jac$(X)$ with multiplicity $\#\mathcal{O}$. 

\par In \cite{kraft}, Kraft shows that $H^1_{dR}(X)$ can be decomposed uniquely into irreducible elements based on the actions of $F$ and $V$ on $H^1_{dR}(X)$. We describe the actions of $F$ and $V$ on $M$ below based on the Frobenius orbits:
\begin{theorem}[\protect{\cite[Definition 2.2]{lin_abelian_2024}}]
\label{eo}
    Let $\tau \in \mathcal{T}_G$ and identify it with an element of $G$. Let $\tau^\ast$ represent the conjugate transpose of $\tau$.  If $M_\tau \neq 0$, there exists an $\overline{\mathbf{F}_p}$-basis $\{e_{\tau} \}$ of $M_\tau$ such that the actions of $F$ and $V$ restricted to $M_\tau$ are given by
    \begin{align*}F(e_{\tau}) &= \begin{cases}
        e_{p\tau} & f(\tau^\ast) = 1 \\ 0 & f(\tau^\ast) = 0,
    \end{cases} & V(e_{p\tau}) &= \begin{cases}
        0 & f(\tau^\ast) = 1 \\
        e_{\tau} & f(\tau^\ast) = 0.
    \end{cases} \end{align*}
\end{theorem}
\par The Ekedahl-Oort type is $\bigoplus_{\tau \in \mathcal{T}_G} M_\tau$.
\section{Computational Results}
\par \hspace{\parindent} In this section, we use SageMath and the results in \cite[Section 2.1]{booher_supersingular_2025} to generate all possible ramification and inertia types of abelian covers of $\mathbf{P}^1$ branched at three points of a given genus. We then compute the Newton polygons and Ekedahl-Oort types for these curves using the methods in sections 2.3 and 2.4. For a given genus, we then have all possible Newton polygons and Ekedahl-Oort types that can occur for covers of this form. This allows us to analyze exactly when supersingular curves, superspecial curves, unlikely Newton polygons, and unlikely Ekedahl-Oort types occur. We see that these curves occur far more often than initially thought. The code is available on Github \cite{ds}.

\subsection{Density Computations}

\begin{definition}
    \label{ssg}
    Let $SS_g$ be the set of primes $p$ such that there exists a supersingular curve $X$ of genus $g$ such that $X$ is an abelian cover of $\mathbf{P}^1$ branched at three points defined over $\overline{\mathbf{F}_p}$. Let $\delta_{ss}(g)$ denote the natural density of the primes in $SS_g$.
\end{definition}

\par Booher and Pries computed $\delta_{ss}(g)$ when $5 \leq g \leq 10$ in \cite[Corollary 1.1]{booher_supersingular_2025}. We extend the results here for $11 \leq g \leq 50$ and record them in Table 1. We note that if $g \geq 9$, a supersingular curve has an unlikely Newton polygon and is geometrically unexpected. However, in Table 1 we see that $\delta_{ss}(g) > 0.7$.

\begin{definition}
    \label{sspg}
    Let $SSP_g$ be the set of primes $p$ such that there exists a superspecial curve $X$ of genus $g$ such that $X$ is an abelian cover of $\mathbf{P}^1$ branched at three points defined over $\overline{\mathbf{F}_p}$. Let $\delta_{ssp}(g)$ denote the natural density of the primes in $SSP_g$.
\end{definition}

\par The superspecial primes for $g = 4,5$ were computed in \cite{ekedahl_supersingular_1987}. We give the value (or an underestimate of the value) of $\delta_{ssp}(g)$ for $5 \leq g \leq 50$ in Table 1.

\par The natural densities in Definitions \ref{ssg} and \ref{sspg} exist as the set of primes is given by congruence conditions.

\par For a given cover, if $e$ is the exponent of the Galois group, we get a list of congruence classes in $\mathbf{Z}/e\mathbf{Z}$ that categorize which primes give us the type of curve that we're interested in. To compute the density of those primes, we compute the least common multiple of all the exponents that occur $\ell$. We then use the Chinese Remainder Theorem to determine which congruence classes modulo $\ell$ that are coprime to $\ell$ satisfy the congruence conditions for any of the covers found. The bottleneck is that as the genus grows larger, there are generally more covers with larger Galois groups. This causes $\ell$ to grow very large very fast, so it becomes infeasible to keep track of all congruence classes modulo $\ell$. We have denoted which computations for which this happens in the table so that the reader is aware they are only estimates of the true value. 

\par From our computation, we noticed an omission in Ekedahl's list of primes for which there exists a superspecial curve of genus $5$ that is an abelian cover of $\mathbf{P}^1$ branched at three points. We correct it below.

\begin{proposition}\label{eke}
    There exists an abelian cover $\phi: X \to \mathbf{P}^1$ of genus $5$ branched at three points in characteristic $p$ such that $X$ is superspecial if and only if $p \equiv -1$ modulo $8,11,12,15,20$, $p \equiv 11$ modulo $15$, or $p \equiv 11$ modulo $20$.
\end{proposition}

\begin{proof}
    \par The only case not proven in \cite[Page 173]{ekedahl_supersingular_1987} is when $p \equiv 11$ modulo $20$. Consider the cover with ramification type $[20,20,2,1]$ and inertia type $(1,9,10)$. This cover has genus $5$ by \cite[Lemma 2.3]{booher_supersingular_2025}. The orbits of the action of multiplication by $p$ on $\{1,2, \ldots, 19\} \subseteq \mathbf{Z}/20\mathbf{Z}$ are as follows:
    \[\{\{1,11\}, \{2\}, \{3,13\}, \{4\}, \{5,15\}, \{6\}, \{7,17\}, \{8\}, \{9,19\}, \{10\}, \{12\}, \{14\}, \{16\}, \{18\} \}.\]
    \par The signature type $f(j) =1$ if and only if $j \in \{1,3,5,7,9\}$. Therefore, by Theorem \ref{eo} $F(e_{11}) = V(e_{11}) = e_1$, $F(e_{13}) = V(e_{13}) = e_3$, $F(e_{15}) =V(e_{15}) = e_5$, $F(e_{17}) = V(e_{17}) = e_7$, and $F(e_{19}) = V(e_{19}) = e_9$. We see that $H^1_{dR}(X)$ is the direct sum of 5 copies of $H^1_{dR}$ of a supersingular elliptic curve. Therefore, Jac$(X)$ is isomorphic to a product of $5$ supersingular elliptic curves, and $X$ is superspecial.
\end{proof}

\subsection{Unlikely Newton Polygons and Ekedahl-Oort Types}

\par \hspace{\parindent} We also investigates the prime density for covers that have unlikely Newton polygons and Ekedahl-Oort types. We give some theorems on the dimensions of some relevant moduli spaces first.
\begin{theorem}[\protect{\cite[Theorem 6.10]{pries_torelli}}]
    \label{npdim}

    Let $\xi$ be a symmetric Newton polygon of height $2g$. The dimension of $\mathcal{A}_g[\xi]$
is $\#\{(x, y) \in \mathbf{Z} \times \mathbf{Z} | y < x \leq g \textrm{ and } (x, y) \textrm{ lies on or above } \xi \}$.
\end{theorem}

\begin{theorem}[\protect{\cite[Theorem 6.11]{pries_torelli}}]
    \label{eodim}
    Let $\xi$ be an Ekedahl-Oort Type represented as the final type $[\nu_1,\nu_2,\ldots,\nu_g]$ . The dimension of $\mathcal{A}_g[\xi]$ is $\sum_{i=1}^g \nu_i$.
\end{theorem}
We use these theorems to compute whether a Newton polygon or Ekedahl-Oort type resulting from an abelian cover is unlikely. We now define the following densities:

\begin{definition}
    Let $NU_g$ (respectively, $EU_g$) denote the set of primes $p$ such that there exists an abelian cover of $\mathbf{P}^1$ branched at three points of genus $g$ defined over $\overline{\mathbf{F}_p}$ that has an unlikely Newton polygon (respectively, unlikely Ekedahl-Oort type). Let $\delta_{nu}(g)$ (respectively, $\delta_{eu}(g)$) denote the natural density of $NU_g$ (respectively, $EU_g$).
\end{definition}

\begin{definition}
    Let $2NU_g$ (respectively, $2EU_g$) denote the set of primes $p$ such that there exists an abelian cover of $\mathbf{P}^1$ branched at three points of genus $g$ defined over $\overline{\mathbf{F}_p}$ with Newton polygon (respectively, Ekedahl-Oort type) $\xi$ such that the codimension of $\mathcal{A}_g[\xi]$ in $\mathcal{A}_g$ is greater than $2$ times the dimension of $\mathcal{M}_g$. Let $\delta_{2nu}(g)$ (respectively, $\delta_{2eu}(g)$) denote the natural density of $2NU_g$ (respectively, $2EU_g$).
\end{definition}

\par The choice of $2$ for $2NU_g$ and $2EU_g$ is somewhat arbitrary; we wanted to show that unlikely Newton polygons and Ekedahl-Oort types still occur even if the codimension of $\mathcal{A}_g[\xi]$ is much larger than the dimension of $\mathcal{M}_g$ to illustrate the robustness of the phenomena.

\par The density $\delta_{nu}(g)$ is in Table $1$ while $\delta_{2nu}(g)$, $\delta_{eu}(g)$, and $\delta_{2eu}(g)$ are in Table $2$. From the data in the table, if the genus is large enough the numerator of each of these densities is always one less than the denominator. Recall a curve is ordinary if its Newton polygon only has slopes $0$ or $1$ or if the only words appearing in its Ekedahl-Oort type are f and v. The strata of an ordinary abelian variety has codimension $0$ in $\mathcal{A}_g$, so the Newton polygon is not unlikely. If $\phi: X \to \mathbf{P}^1$ is an abelian cover of degree $m$ branched at three points, then if $p$ is a prime with $p \equiv 1$ modulo $m$, then $X$ will be an ordinary curve. This is because the multiplication by $p$ action on the elements of the Galois group of $\phi$ give orbits that are singletons, so by the Shimura-Taniyama formula the Newton polygon of $X$ can only have slopes $0$ and $1$. The data show that as long as the genus $g$ is large enough and $g \leq 50$, for every prime $p$ used in the computation that is not congruent to $1$ modulo the order of any Galois group of a cover of genus $g$, there exists a cover defined over $\overline{\mathbf{F}_p}$ that has an unlikely Newton polygon or unlikely Ekedahl-Oort type. This is led us to make Conjecture \ref{unlikely-density}.

\begin{table}
\centering

\begin{adjustbox}{width=\textwidth}
\begin{tabular}{|c|c|c|c|c|c|c|} \hline
$g$ & $\delta_{ss}(g)$ & $\approx \delta_{ss}(g)$ & $\delta_{ssp}(g)$ & $\approx \delta_{ssp}(g)$ & $\delta_{nu}(g)$ & $\approx \delta_{nu}(g)$ \\ \hline
$5$ & 25/32 & $0.7812$ & 97/160 & $0.6062$ & 0 & $0.0000$ \\ \hline
$6$ & 507/512 & $0.9902$ & 347/512 & $0.6777$ & 0 & $0.0000$ \\ \hline
$7$ & 3/4 & $0.7500$ & 29/48 & $0.6042$ & 0 & $0.0000$ \\ \hline
$8$ & 1023/1024 & $0.9990$ & 589/1024 & $0.5752$ & 0 & $0.0000$ \\ \hline
$9$ & 15/16 & $0.9375$ & 871/1296 & $0.6721$ & 15/16 & $0.9375$ \\ \hline
$10$ & 31/32 & $0.9688$ & 131/200 & $0.6550$ & 31/32 & $0.9688$ \\ \hline
$11$ & 7/8 & $0.8750$ & 19/40 & $0.4750$ & 7/8 & $0.8750$ \\ \hline
$12$ & 1019/1024 & $0.9951$ & 15713/27648 & $0.5683$ & 3067/3072 & $0.9984$ \\ \hline
$13$ & 183/256 & $0.7148$ & 3709/6912 & $0.5366$ & 631/768 & $0.8216$ \\ \hline
$14$ & 1001/1024 & $0.9775$ & 3415/7168 & $0.4764$ & 3573/3584 & $0.9969$ \\ \hline
$15$ & 121/128 & $0.9453$ & 1079/1920 & $0.5620$ & 4319/4320 & $0.9998$ \\ \hline
$16$ & 8189/8192 & $0.9996$ & 4547/8192 & $0.5551$ & 40957/40960 & $0.9999$ \\ \hline
$17$ & 65/128 & $0.5078$ & 283/768 & $0.3685$ & 371/384 & $0.9661$ \\ \hline
$18$ & 2045/2048 & $0.9985$ & 563419/1119744 & $0.5032$ & 6718463/6718464 & $>0.9999$ \\ \hline
$19$ & 3/4 & $0.7500$ & 697/1296 & $0.5378$ & 6911/6912 & $0.9999$ \\ \hline
$20$ & 1021/1024 & $0.9971$ & 3437/8000 & $0.4296$ & 127999/128000 & $>0.9999$ \\ \hline
$21$ & 353/384 & $0.9193$ & 121151/282240 & $0.4292$ & 423359/423360 & $>0.9999$ \\ \hline
$22$ & 61/64 & $0.9531$ & 7751/17600 & $0.4404$ & 3839/3840 & $0.9997$ \\ \hline
$23$ & 7/8 & $0.8750$ & 3907/16192 & $0.2413$ & 4047/4048 & $0.9998$ \\ \hline
$24$ & 32763/32768 & $0.9998$ & \textgreater102447/229376 & $0.4466$ & 7741439/7741440 & $>0.9999$ \\ \hline
$25$ & 17529/20480 & $0.8559$ & 10083/20480 & $0.4923$ & 921599/921600 & $>0.9999$ \\ \hline
$26$ & 497/512 & $0.9707$ & 7303/19968 & $0.3657$ & 99839/99840 & $>0.9999$ \\ \hline
$27$ & 121/128 & $0.9453$ & 87493/207360 & $0.4219$ & 933119/933120 & $>0.9999$ \\ \hline
$28$ & 4061/4096 & $0.9915$ & 1937/4096 & $0.4729$ & 290303/290304 & $>0.9999$ \\ \hline
$29$ & 405/512 & $0.7910$ & 28335/103936 & $0.2726$ & 38975/38976 & $>0.9999$ \\ \hline
$30$ & \textgreater2549/2560 & $0.9957$ & \textgreater1365281/3686400 & $0.3704$ & \textgreater15551999/15552000 & $>0.9999$ \\ \hline
$31$ & 25/32 & $0.7812$ & 4073/9600 & $0.4243$ & 345599/345600 & $>0.9999$ \\ \hline
$32$ & 4095/4096 & $0.9998$ & 2927/8192 & $0.3573$ & 8847359/8847360 & $>0.9999$ \\ \hline
$33$ & \textgreater1973/2048 & $0.9634$ & \textgreater45389/139392 & $0.3256$ & \textgreater8363519/8363520 & $>0.9999$ \\ \hline
$34$ & 1497/2048 & $0.7310$ & 25565/67584 & $0.3783$ & 168959/168960 & $>0.9999$ \\ \hline
$35$ & 439/512 & $0.8574$ & 80319/224000 & $0.3586$ & 1007999/1008000 & $>0.9999$ \\ \hline
$36$ & \textgreater1023/1024 & $0.9990$ & \textgreater47497/93312 & $0.5090$ & \textgreater13436927/13436928 & $>0.9999$ \\ \hline
$37$ & \textgreater183/256 & $0.7148$ & \textgreater102941/311040 & $0.3310$ & \textgreater5598719/5598720 & $>0.9999$ \\ \hline
$38$ & 239/256 & $0.9336$ & 26683/103680 & $0.2574$ & 207359/207360 & $>0.9999$ \\ \hline
$39$ & 2023/2304 & $0.8780$ & 236209/808704 & $0.2921$ & 4852223/4852224 & $>0.9999$ \\ \hline
$40$ & \textgreater81871/81920 & $0.9994$ & \textgreater4201/10240 & $0.4103$ & \textgreater6911999/6912000 & $>0.9999$ \\ \hline
$41$ & \textgreater197/256 & $0.7695$ & \textgreater42833/131200 & $0.3265$ & \textgreater1574399/1574400 & $>0.9999$ \\ \hline
$42$ & \textgreater287/288 & $0.9965$ & \textgreater67/256 & $0.2617$ & \textgreater5419007/5419008 & $>0.9999$ \\ \hline
$43$ & 3/4 & $0.7500$ & 117/448 & $0.2612$ & 564479/564480 & $>0.9999$ \\ \hline
$44$ & \textgreater1021/1024 & $0.9971$ & \textgreater4923/17600 & $0.2797$ & \textgreater4646399/4646400 & $>0.9999$ \\ \hline
$45$ & \textgreater1117/1152 & $0.9696$ & \textgreater788639/2073600 & $0.3803$ & \textgreater13996799/13996800 & $>0.9999$ \\ \hline
$46$ & \textgreater15/16 & $0.9375$ & \textgreater35659/105600 & $0.3377$ & \textgreater6071999/6072000 & $>0.9999$ \\ \hline
$47$ & \textgreater51/64 & $0.7969$ & \textgreater59071/264960 & $0.2229$ & \textgreater794879/794880 & $>0.9999$ \\ \hline
$48$ & \textgreater262141/262144 & $>0.9999$ & \textgreater565847/2064384 & $0.2741$ & \textgreater15482879/15482880 & $>0.9999$ \\ \hline
$49$ & \textgreater25/32 & $0.7812$ & \textgreater110717/286720 & $0.3862$ & \textgreater7741439/7741440 & $>0.9999$ \\ \hline
$50$ & \textgreater63/64 & $0.9844$ & \textgreater3911/14400 & $0.2716$ & \textgreater5759999/5760000 & $>0.9999$ \\ \hline

\end{tabular}

\end{adjustbox}

\caption{Supersingular Covers, Superspecial Covers, and Covers with Unlikely Newton Polygons}
\end{table}

\begin{table}
\centering

\begin{adjustbox}{width=\textwidth}
\begin{tabular}{|c|c|c|c|c|c|c|} \hline
$g$ & $\delta_{2nu}(g)$ & $\approx \delta_{2nu}(g)$ & $\delta_{eu}(g)$ & $\approx \delta_{eu}(g)$ & $\delta_{2eu}(g)$ & $\approx \delta_{2eu}(g)$ \\ \hline
$5$ & 0 & $0.0000$ & 139/160 & $0.8688$ & 0 & $0.0000$ \\ \hline
$6$ & 0 & $0.0000$ & 507/512 & $0.9902$ & 0 & $0.0000$ \\ \hline
$7$ & 0 & $0.0000$ & 7/8 & $0.8750$ & 0 & $0.0000$ \\ \hline
$8$ & 0 & $0.0000$ & 2045/2048 & $0.9985$ & 0 & $0.0000$ \\ \hline
$9$ & 0 & $0.0000$ & 3887/3888 & $0.9997$ & 0 & $0.0000$ \\ \hline
$10$ & 0 & $0.0000$ & 5759/5760 & $0.9998$ & 69/80 & $0.8625$ \\ \hline
$11$ & 0 & $0.0000$ & 879/880 & $0.9989$ & 41/55 & $0.7455$ \\ \hline
$12$ & 0 & $0.0000$ & 34559/34560 & $>0.9999$ & 2291/2304 & $0.9944$ \\ \hline
$13$ & 0 & $0.0000$ & 1727/1728 & $0.9994$ & 559/768 & $0.7279$ \\ \hline
$14$ & 0 & $0.0000$ & 2687/2688 & $0.9996$ & 2651/2688 & $0.9862$ \\ \hline
$15$ & 0 & $0.0000$ & 172799/172800 & $>0.9999$ & 85861/86400 & $0.9938$ \\ \hline
$16$ & 0 & $0.0000$ & 61439/61440 & $>0.9999$ & 20443/20480 & $0.9982$ \\ \hline
$17$ & 0 & $0.0000$ & 575/576 & $0.9983$ & 89/96 & $0.9271$ \\ \hline
$18$ & 0 & $0.0000$ & 13436927/13436928 & $>0.9999$ & 1679615/1679616 & $>0.9999$ \\ \hline
$19$ & 0 & $0.0000$ & 20735/20736 & $>0.9999$ & 431/432 & $0.9977$ \\ \hline
$20$ & 0 & $0.0000$ & 1151999/1152000 & $>0.9999$ & 63999/64000 & $>0.9999$ \\ \hline
$21$ & 353/384 & $0.9193$ & 1693439/1693440 & $>0.9999$ & 21167/21168 & $>0.9999$ \\ \hline
$22$ & 61/64 & $0.9531$ & 316799/316800 & $>0.9999$ & 6329/6336 & $0.9989$ \\ \hline
$23$ & 7/8 & $0.8750$ & 48575/48576 & $>0.9999$ & 97111/97152 & $0.9996$ \\ \hline
$24$ & \textgreater229357/229376 & $0.9999$ & \textgreater15482879/15482880 & $>0.9999$ & \textgreater4423679/4423680 & $>0.9999$ \\ \hline
$25$ & 97757/102400 & $0.9547$ & 16588799/16588800 & $>0.9999$ & 921599/921600 & $>0.9999$ \\ \hline
$26$ & 19931/19968 & $0.9981$ & 10782719/10782720 & $>0.9999$ & 59903/59904 & $>0.9999$ \\ \hline
$27$ & 1909/1920 & $0.9943$ & 1866239/1866240 & $>0.9999$ & 466559/466560 & $>0.9999$ \\ \hline
$28$ & 773989/774144 & $0.9998$ & 1161215/1161216 & $>0.9999$ & 290303/290304 & $>0.9999$ \\ \hline
$29$ & 51867/51968 & $0.9981$ & 467711/467712 & $>0.9999$ & 38975/38976 & $>0.9999$ \\ \hline
$30$ & \textgreater2879797/2880000 & $0.9999$ & \textgreater15551999/15552000 & $>0.9999$ & \textgreater6911999/6912000 & $>0.9999$ \\ \hline
$31$ & 28681/28800 & $0.9959$ & 2764799/2764800 & $>0.9999$ & 345599/345600 & $>0.9999$ \\ \hline
$32$ & 61439/61440 & $>0.9999$ & 8847359/8847360 & $>0.9999$ & 8847359/8847360 & $>0.9999$ \\ \hline
$33$ & \textgreater2230271/2230272 & $>0.9999$ & \textgreater16727039/16727040 & $>0.9999$ & \textgreater8363519/8363520 & $>0.9999$ \\ \hline
$34$ & 202619/202752 & $0.9993$ & 3041279/3041280 & $>0.9999$ & 168959/168960 & $>0.9999$ \\ \hline
$35$ & 671917/672000 & $0.9999$ & 2015999/2016000 & $>0.9999$ & 335999/336000 & $>0.9999$ \\ \hline
$36$ & \textgreater3359231/3359232 & $>0.9999$ & \textgreater20155391/20155392 & $>0.9999$ & \textgreater6718463/6718464 & $>0.9999$ \\ \hline
$37$ & \textgreater34549/34560 & $0.9997$ & \textgreater5598719/5598720 & $>0.9999$ & \textgreater2799359/2799360 & $>0.9999$ \\ \hline
$38$ & 69107/69120 & $0.9998$ & 1244159/1244160 & $>0.9999$ & 207359/207360 & $>0.9999$ \\ \hline
$39$ & 269563/269568 & $>0.9999$ & 9704447/9704448 & $>0.9999$ & 4852223/4852224 & $>0.9999$ \\ \hline
$40$ & \textgreater2303993/2304000 & $>0.9999$ & \textgreater16588799/16588800 & $>0.9999$ & \textgreater6911999/6912000 & $>0.9999$ \\ \hline
$41$ & \textgreater6296393/6297600 & $0.9998$ & \textgreater6297599/6297600 & $>0.9999$ & \textgreater1574399/1574400 & $>0.9999$ \\ \hline
$42$ & \textgreater5418991/5419008 & $>0.9999$ & \textgreater10838015/10838016 & $>0.9999$ & \textgreater5419007/5419008 & $>0.9999$ \\ \hline
$43$ & 3919/3920 & $0.9997$ & 1128959/1128960 & $>0.9999$ & 564479/564480 & $>0.9999$ \\ \hline
$44$ & 619517/619520 & $>0.9999$ & \textgreater18585599/18585600 & $>0.9999$ & \textgreater4646399/4646400 & $>0.9999$ \\ \hline
$45$ & \textgreater5598719/5598720 & $>0.9999$ & \textgreater27993599/27993600 & $>0.9999$ & \textgreater13996799/13996800 & $>0.9999$ \\ \hline
$46$ & \textgreater1011999/1012000 & $>0.9999$ & \textgreater12143999/12144000 & $>0.9999$ & \textgreater6071999/6072000 & $>0.9999$ \\ \hline
$47$ & \textgreater26495/26496 & $>0.9999$ & \textgreater1589759/1589760 & $>0.9999$ & \textgreater794879/794880 & $>0.9999$ \\ \hline
$48$ & \textgreater1935359/1935360 & $>0.9999$ & \textgreater46448639/46448640 & $>0.9999$ & \textgreater7741439/7741440 & $>0.9999$ \\ \hline
$49$ & \textgreater806399/806400 & $>0.9999$ & \textgreater7741439/7741440 & $>0.9999$ & \textgreater7741439/7741440 & $>0.9999$ \\ \hline
$50$ & \textgreater1919999/1920000 & $>0.9999$ & \textgreater5759999/5760000 & $>0.9999$ & \textgreater5759999/5760000 & $>0.9999$ \\ \hline
\end{tabular}

\end{adjustbox}
\caption{Covers with Unlikely Newton polygons and Unlikely Ekedahl-Oort Types}

\end{table}

\section{Unlikely Newton Polygons}

\par \hspace{\parindent}Using the algorithm to compute Newton polygons of covers, we found repeated occurences of certain kinds of Newton polygons, including supersingular Newton polygons and unlikely Newton polygons with large denominators. In this section, we prove Theorem \ref{slopes}, which formalizes the patterns we saw in the data. We then prove the corollaries that follow from the theorem.

\par Before proving Theorem \ref{slopes}, we state a specific counterexample to Expectation \ref{expectation1} that does not arise from Theorem \ref{slopes}.

\subsection{A New Newton Polygon}

\par \hspace{\parindent} Pries notes in \cite[Section 5.3]{pries_newtonpolygons} that a curve of genus $12$ with Newton polygon that only has slopes $5/12$ and $7/12$ is only known to exist in characteristic $2$. Expectation \ref{expectation1} says that no curve with this Newton polygon exists. Through the Shimura-Taniyama method we find that this Newton polygon also occurs for a curve in some odd characteristics.

\begin{proposition}
    There is a curve with Newton polygon in characteristic $p$ with slopes $5/12$ and $7/12$, contrary to Expectation \ref{expectation1}, when $p \equiv 3,12,17,33$ modulo $35$. 
\end{proposition}

\par We demonstrate this when $p \equiv 3$ modulo $35$.

\begin{proof}
    The genus 12 curve has the equation 
    \[y^{35} = x(1-x)^{20}.\]
    
    The ramification type and inertia type of this curve is $[35,7,5,1]$ and $(1,20,14)$, respectively. If $p \equiv 3$ modulo $35$, then the orbits of the action of multiplication by $p$ on $\{1,2,3,\ldots,34\}$ are
    \begin{align*}\{&\{1,3,9,27,11,33,29,17,16,13,4,12\}, \{2,6,18,19,22,31,23,34,32,26,8,24\} , \\ \{&5,15,10,30,20,25\}, \{7,21,28,14\}\}
    \end{align*}
    
    As $5$ and $7$ have orders $7$ and $5$, respectively, they aren't relevant in the Shimura-Taniyama method. We then see that the signature type $f(j)=1$ if and only if $j \in \{1,2,3,4,6,8,9,11,13,16,18,23\}$. Therefore, the first orbit has $7$ elements in it with signature type $1$ and the second orbit has $5$ elements in it with signature type 1, so the Shimura-Taniyama formula tells us that the Newton polygon only has slopes $5/12$ and $7/12$.
\end{proof}

\subsection{Quadratic Excesses and Newton Polygons of Covers} \label{qe}
\par \hspace{\parindent} In order to prove Theorem \ref{slopes}, we define the quadratic excess of an interval.

\begin{definition}
    Let $\ell$ be a prime and $a,b \in \mathbf{R}$ with $a < b$. The \textit{quadratic excess} of the interval $(a,b)$ is \[q(a,b) := \sum_{a < n < b} \left (\frac{n}{\ell} \right )\]
\end{definition}

\par The quadratic excess of an interval is the number of non-zero quadratic residues minus the number of quadratic nonresidues in the interval.

\par We prove a lemma that will be useful in the proof of the theorem.

\begin{lemma}
\label{lemma}
    Let $\ell > 3$ be a prime. If $n$ and $r$ are positive integers such that $n \geq r$, then \begin{align}\left (\frac{-1}{\ell} \right ) q\left(\frac{\ell r}{4n}, \frac{4\ell r}{4n} \right) = \sum_{i=1}^3 q\left (\frac{\ell(in-r)}{4n}, \frac{i\ell n}{4n} \right) \label{lemma-align}. \end{align}
\end{lemma}

\begin{proof}

    \par First, we note that
    
    \begin{align}q\left(\frac{\ell(n+r-1)}{4n}, \frac{\ell(n+r)}{4n} \right) = \left( \frac{-1}{\ell} \right) q \left(\frac{\ell(3n-r)}{4n}, \frac{\ell(3n-r+1)}{4n} \right). \label{legendre} \end{align}

    \par We proceed by induction. If $r = 1$, then by \cite[Main Theorem]{johnson_symmetries_1977} 
    
    \begin{align} \begin{split} \left(\frac{4}{\ell} \right)q \left(0, \frac{\ell}{n} \right) &=  q \left(0, \frac{\ell}{n} \right) =q \left(0, \frac{\ell}{4n} \right) + \left(\frac{-1}{\ell} \right)q\left(\frac{\ell(n-1)}{4n}, \frac{\ell n}{4n} \right) \\ &+ q \left(\frac{\ell n}{4n}, \frac{\ell(n+1)}{4n} \right) + \left(\frac{-1}{\ell} \right)q\left(\frac{\ell(2n-1)}{4n}, \frac{2\ell n}{4n} \right). \\
    \end{split} \label{jmthm} \end{align}
    
    \par Substitute equation (\ref{legendre}) with $r=1$ into equation (\ref{jmthm}), and after rearranging we get \begin{align*}
        \left(\frac{-1}{\ell} \right)q\left(\frac{\ell}{4n}, \frac{\ell}{n} \right) &= q\left(\frac{\ell(n-1)}{4n}, \frac{\ell n}{4n} \right) + q\left(\frac{\ell(2n-1)}{4n}, \frac{2\ell n}{4n} \right) + q\left(\frac{\ell(3n-1)}{4n}, \frac{3\ell n}{4n} \right),
    \end{align*}
    which is equation (\ref{lemma-align}) when $r = 1$.
    \par Now, assume that for $r = a < n$, equation (\ref{lemma-align}) holds. Then by \cite[Main Theorem]{johnson_symmetries_1977} we have that \begin{align} \begin{split} q \left (\frac{\ell a}{n}, \frac{\ell(a+1)}{n} \right ) &= q\left(\frac{\ell a}{4n} , \frac{\ell(a+1)}{4n} \right)  + \left(\frac{-1}{\ell} \right)q \left(\frac{\ell(n-a-1)}{4n}, \frac{\ell(n-a)}{4n} \right) \\ &+ q\left(\frac{\ell(n+a)}{4n}, \frac{\ell(n+a+1)}{4n} \right) + \left (\frac{-1}{\ell} \right) q \left(\frac{\ell(2n-a-1)}{4n}, \frac{\ell(2n-a)}{4n} \right) \end{split} \label{inductive} \end{align}

    \par Now, substitute equation (\ref{legendre}) with $r = a+1$ into equation (\ref{inductive}) and multiply both sides by $\left ( \frac{-1}{\ell}\right)$ and we get \begin{align} \label{lemma-final}
        \left (\frac{-1}{\ell} \right)\left( q\left(\frac{4\ell a}{4n}, \frac{4\ell(a+1)}{4n} \right) - q\left(\frac{\ell a}{4n}, \frac{\ell(a+1)}{4n} \right) \right) = \sum_{i=1}^3 q\left(\frac{\ell(in-a-1)}{4n}, \frac{\ell(in-a)}{4n} \right)
    \end{align}
    
    \par Add equation (\ref{lemma-align}) with $r = a$ and equation (\ref{lemma-final}) together and the result holds for $a+1$.
\end{proof}

\par The proposition below describes the orbits of the Galois group by the action of multiplication by a prime $\ell$. This will be essential in utilizing the Shimura-Taniyama formula to compute the slopes of the Newton polygon.

\begin{proposition}
\label{orbits}
    Let $\ell$ be a prime and suppose $n \geq 1$ is coprime to $\ell$. Let $p$ be a prime such that $p$ has order $g := (\ell-1)/2$ in $(\mathbf{Z}/\ell\mathbf{Z})^\times$, gcd($n\ell$,$p) = 1$, and the order of $p$ in $(\mathbf{Z}/n\mathbf{Z})^\times$ is coprime to $g$. Define $A := \{j \in \mathbf{Z}/(n\ell)\mathbf{Z} \mid j \neq 0\}$, and consider the action of multiplying elements of $A$ by $\ell$. If $a,b \in A$ with $a \equiv b$ modulo $n$ and $\left (\frac{a}{\ell} \right) = \left (\frac{b}{\ell} \right)$, then $a$ and $b$ are in the same orbit under this action. Furthermore, if $a,b \in A$ are in the same orbit, then $\left (\frac{a}{\ell} \right) = \left (\frac{b}{\ell} \right)$.
\end{proposition}

\begin{proof}
\par Since $p$ has order $g = \frac{\ell-1}{2}$ in $(\mathbf{Z}/\ell\mathbf{Z})^\times$, it is a quadratic residue modulo $\ell$. By properties of the Legendre symbol, we have that if $j \in A$, then $\left (\frac{jp}{\ell} \right)  = \left (\frac{j}{\ell} \right)$. Therefore, if $a,b \in A$ are in the same orbit, then $\left(\frac{a}{\ell} \right) = \left(\frac{b}{\ell} \right)$.

\par Now, assume $a,b \in A$ such that $a \equiv b$ modulo $n$ and $\left(\frac{a}{\ell} \right) = \left (\frac{b}{\ell} \right)$. If $\left ( \frac{a}{\ell} \right) = 0$, then $a$ and $b$ are both multiples of $\ell$, so $a-b \equiv 0 $ modulo $n$ and $a-b \equiv 0$ modulo $\ell$, so $a = b$ by the Chinese Remainder Theorem.

\par We note that by the Chinese Remainder Theorem, if $c$ is the order of $p$ in $\left(\mathbf{Z}/n\mathbf{Z} \right)^\times$, then since $c$ is coprime to $g$, the order of $p$ in $\left(\mathbf{Z}/(n\ell)\mathbf{Z} \right)^\times$ is $gc$. As $a p^k \equiv a$ modulo $n$ if and only if $k$ is a multiple of $c$, and since $\left (\frac{p}{\ell} \right) = 1$, $\left (\frac{ap^k}{\ell} \right) = \left(\frac{a}{\ell}\right )$, there are $g$ elements in the orbit containing $a$ that are congruent to $a$ modulo $n$. As there are exactly $g$ elements in $\{b \in A \mid a \equiv b \textrm{ modulo } n \textrm{ and } \left (\frac{a}{\ell} \right) = \left(\frac{b}{\ell} \right)\}$, if $a \equiv b$ modulo $n$ and they have the same evaluation of the Legendre symbol modulo $\ell$, they are in the same orbit.
\end{proof}

\begin{remark}
    If $a \in A$, then $\mathcal{O}_a$, the orbit containing $a$, consists of every element in $A$ that is in the same congruence class as $a$ modulo $n$ and has Legendre symbol equal to $\left( \frac{a}{\ell} \right)$. Observe that if $\left( \frac{a}{\ell} \right) \neq 0$, then $\#\mathcal{O}_a$ is a multiple of $g$, which will be important when we apply the Shimura-Taniyama formula.
\end{remark}

\begin{construction}
    \label{construction}
    Let $\ell > 3$ be a prime and $n \geq 1$ such that gcd$(n,\ell) = 1$. Pick $0 \leq k \leq n-1$ such that $k \equiv -3\ell^{-1}$ modulo $n$.  Consider the cover $\phi: X \to \mathbf{P}^1$ branched at three points with ramification type $[n\ell,n\ell,\ell,1]$ and inertia type $(1,(n-k)\ell-4,k\ell+3)$. The equation for $X$ is $y^{n\ell} = x(1-x)^{(n-k)\ell-4}$.
\end{construction}

\begin{definition}If $\ell$ is a prime and $n \geq 1$, we define a function $I: \{0,1,2,\ldots,6n-1\} \to \mathcal{P}(\mathbf{\mathbf{Z}})$ by 
\[I(i) = \begin{cases}
    \left[\left \lfloor \frac{\ell\left(\frac{i}{2} \right )}{3} \right \rfloor + 1, \left \lfloor \frac{\ell \left( \frac{i}{2} +  \left \lceil \frac{i+1}{6}\right \rceil \right )}{4} \right \rfloor \right] & i \textrm{ even} \\\\

    \left[ \left \lfloor \frac{\ell \left ( \frac{i-1}{2} + \left \lceil \frac{i}{6}\right \rceil \right)}{4} \right \rfloor + 1,  \left \lfloor \frac{\ell\left( \frac{i+1}{2} \right)}{3}\right \rfloor \right] & i \textrm{ odd}
\end{cases}\]

\end{definition}

Note that $\{I(i) \mid 0 \leq i \leq 6n-1\}$ is a partition of the set of integers in the interval $[1,\ell n-1]$.

\begin{lemma}
\label{sigtype}
    Let $\phi$ be the cover in Construction \ref{construction}, and let $f$ be the signature type of $\phi$. If $j \in \{1,\ldots,\ell n-1\}$ and $j \not \equiv 0$ modulo $\ell$, pick $i$ such that $0 \leq i \leq 6n-1$ and $j \in I(i)$. Then,
    \begin{align*}
        f(j) = \left \lceil \frac{jk+\left \lfloor \frac{i+2}{2}\right \rfloor }{n}\right \rceil - \left \lfloor \frac{jk + \left \lfloor \frac{i+1}{2}\right \rfloor + \left \lfloor \frac{i}{6} \right \rfloor }{n} \right \rfloor
    \end{align*}
\end{lemma}
\begin{proof}

Section 2.2 tells us that 

\begin{align}
f(j) = -1 + \left \langle \frac{j}{n\ell}\right \rangle &+ \left \langle \frac{j(k\ell+3)}{n\ell}\right \rangle + \left \langle \frac{j((n-k)\ell-4)}{n\ell}\right \rangle \\ &= \left \lceil \frac{j(k\ell+3)}{n\ell}\right \rceil - \left \lfloor \frac{j(k\ell+4)}{n\ell}\right \rfloor \label{floor}
\end{align}

with equation (\ref{floor}) following from the definition of the fractional part of a real number.
\\\\
\par Given $j \in \{1,2,\ldots,\ell n-1\}$ with $j \not \equiv 0 $ modulo $\ell$ and the $i$ that determines the interval $j$ lies in, we will prove the following: 
\begin{align}
\left \lceil \frac{j(k\ell+3)}{n\ell}\right \rceil &= \left \lceil \frac{jk+ \lfloor \frac{i+2}{2} \rfloor}{n}\right \rceil  \label{first-floor} \\ \left \lfloor \frac{j(k\ell+4)}{n\ell}\right \rfloor &= \left \lfloor \frac{jk + \left \lfloor \frac{i+1}{2}\right \rfloor + \left \lfloor \frac{i}{6} \right \rfloor }{n} \right \rfloor  \label{second-floor}
\end{align}

\par To prove equation (\ref{first-floor}), we note that if $i$ is even, then 

\begin{align*}
    \frac{3j}{n\ell} \leq \frac{3\ell\left (\frac{i}{2} + \left \lceil \frac{i+1}{6} \right \rceil\right)}{4n\ell} \leq \frac{3\ell\left(\frac{i}{2} + \frac{i+5}{6} \right)}{4n\ell} = \frac{\frac{i+2}{2}}{n} - \frac{\ell}{8n\ell} < \frac{\left \lfloor \frac{i+2}{2}\right \rfloor}{n}
\end{align*}

\par and

\begin{align*}
    \frac{3j}{n\ell} \geq \frac{\ell \left ( \frac{i}{2} \right) - 2}{n\ell} + \frac{3}{n\ell} = \frac{\frac{i+2}{2} }{n} + \frac{1-\ell}{n\ell} > \frac{\left \lfloor \frac{i+2}{2} \right \rfloor}{n} - \frac{1}{n}.
\end{align*}

\par If $i$ is odd, then

\begin{align*}
    \frac{3j}{n\ell} \leq \frac{\ell \left( \frac{i+1}{2} \right)}{n\ell} = \frac{ \frac{i+1}{2}}{n}  = \frac{\left \lfloor \frac{i+2}{2} \right \rfloor}{n}
\end{align*}

\par and

\begin{align*}
    \frac{3j}{n\ell} \geq \frac{3\ell \left(\frac{i-1}{2} + \frac{i}{6} \right) - 9}{4n\ell} + \frac{3}{n\ell} = \frac{ \frac{i+1}{2}}{n} + \frac{\frac{-21\ell}{6}+3}{4n\ell} > \frac{\left \lfloor \frac{i+2}{2} \right \rfloor}{n} -\frac{1}{n}
\end{align*}

\par For equation (\ref{second-floor}), if $i$ is even, note that $\left \lfloor \frac{i}{6} \right \rfloor + 1 = \left \lceil \frac{i+1}{6} \right \rceil$ and $\frac{i}{2} + \left \lceil \frac{i+1}{6} \right \rceil$ is never $0$ modulo $4$. Then,

\begin{align*}
    \frac{4j}{n\ell} \leq \frac{\ell\left( \frac{i}{2} + \left \lceil \frac{i+1}{6} \right \rceil \right)-1}{n\ell} = \frac{\left \lfloor \frac{i+1}{2}\right \rfloor + \left \lfloor \frac{i}{6} \right \rfloor}{n}  +\frac{\ell-1}{n\ell} < \frac{\left \lfloor \frac{i+1}{2}\right \rfloor + \left \lfloor \frac{i}{6} \right \rfloor}{n} + \frac{1}{n}
\end{align*}

\par and

\begin{align*}
    \frac{4j}{n\ell} \geq \frac{4\ell\left(\frac{i}{2} \right)-8}{3n\ell}+\frac{4}{n\ell} = \frac{\frac{i}{2} + \frac{i}{6}}{n}+\frac{4}{3n\ell} > \frac{\left \lfloor \frac{i+1}{2}\right \rfloor + \left \lfloor \frac{i}{6} \right \rfloor}{n}
\end{align*}

\par Observe that $j$ is a multiple of $\ell$ if and only if $j$ is the rightmost endpoint of $I(i)$ for some $i \equiv 5$ modulo $6$. Since we are excluding the case where $j \equiv 0$ modulo $\ell$, this gives us that if $i$ is odd, $\frac{i+1}{2}$ is never a multiple of $3$. Also, when $i$ is odd and not $5$ modulo $6$, we have that $\frac{i+1}{6} \leq \left \lfloor \frac{i}{6} \right \rfloor + 1$. We now show that if $i$ is odd,

\begin{align*}
    \frac{4j}{n\ell} \leq \frac{4\ell\left( \frac{i+1}{2} \right)-1}{3n\ell} = \frac{\frac{i+1}{2} + \frac{i+1}{6}}{n}-\frac{1}{3n\ell} \leq \frac{\left \lfloor \frac{i+1}{2}\right \rfloor + \left \lfloor \frac{i}{6} \right \rfloor+1}{n} - \frac{1}{3n\ell} < \frac{\left \lfloor \frac{i+1}{2}\right \rfloor + \left \lfloor \frac{i}{6} \right \rfloor}{n} + \frac{1}{n}
\end{align*}

\par and as $i \not \equiv 0 $ modulo $6$ in this case, $\left \lceil \frac{i}{6} \right \rceil -1 = \left \lfloor \frac{i}{6} \right \rfloor$, so

\begin{align*}
    \frac{4j}{n\ell} \geq \frac{\ell \left(\frac{i-1}{2} + \left \lceil \frac{i}{6} \right \rceil \right) - 3}{n\ell} + \frac{4}{n\ell} = \frac{\left \lfloor \frac{i+1}{2}\right \rfloor + \left \lfloor \frac{i}{6} \right \rfloor}{n} + \frac{\ell+1}{n\ell} > \frac{\left \lfloor \frac{i+1}{2}\right \rfloor + \left \lfloor \frac{i}{6} \right \rfloor}{n}.
\end{align*}
\end{proof}

\begin{lemma}
    \label{intervals}
    Let $\phi$ be the cover in Construction \ref{sigtype}, and let $f$ be the signature type of $\phi$. Let $j \in \{1,2,\ldots,\ell n-1\}$ such that $j \equiv a\ell$ modulo $n$ for some $0 \leq a \leq n-1$ and $j \not \equiv 0$ modulo $\ell$. We know that $f(j) = 1$ if and only if
    \begin{itemize}

    \item $0 \leq 3a \leq n-1$ and
    
    \[
        j \in \left(0, \frac{3a\ell}{4}\right) \cup \left(a\ell, \frac{\ell(n+3a)}{4} \right) \cup \left(\frac{\ell(n+3a)}{3}, \frac{\ell(2n+3a)}{4} \right) \cup \left(\frac{\ell(2n+3a)}{3}, \frac{\ell(3n+3a)}{4} \right),
    \]
    
    \item or $n \leq 3a \leq 2n-1$ and

    \[
        j \in \left(0, \frac{\ell(3a-n)}{4}\right) \cup \left(\frac{\ell(3a-n)}{3}, \frac{3\ell a}{4} \right) \cup  \left(a\ell, \frac{\ell(n+3a)}{4} \right) \cup  \left(\frac{\ell(n+3a)}{3}, \frac{\ell(2n+3a)}{4} \right),
    \]
    
    \item or $2n \leq 3a \leq 3n-1$ and
   
    \[
       j \in \left(0, \frac{\ell(3a-2n)}{4}\right) \cup \left(\frac{\ell(3a-2n)}{3}, \frac{\ell(3a-n)}{4} \right) \cup  \left(\frac{\ell(3a-n)}{3}, \frac{3\ell a}{4} \right) \cup  \left(a\ell, \frac{\ell(n+3a)}{4} \right).
    \]
     \end{itemize}
\end{lemma}

\begin{proof}[]
    If $j \equiv 0$ modulo $\ell$, then by equation (\ref{floor}), $f(j) = 0$. The signature type $f(j) = 0$ if and only if there is a multiple of $n$ between $jk + \left \lfloor \frac{i+2}{2} \right \rfloor + 1$ and $jk + \left \lfloor \frac{i+1}{2} \right \rfloor + \left \lfloor \frac{i}{6} \right \rfloor$. This means that if $j \equiv a\ell$ modulo $n$ for $0 \leq a \leq n-1$, then $f(j) = 1$ if and only if for all $m$ such that $\left \lfloor \frac{i+2}{2} \right \rfloor + 1 \leq m \leq \left \lfloor \frac{i+1}{2} \right \rfloor + \left \lfloor \frac{i}{6} \right \rfloor$, $m \not \equiv 3a$ modulo $n$. Note that $jk \equiv -3a$ modulo $n$, so if $m \equiv 3a$ modulo $n$, then $jk + m \equiv 0$ modulo $n$. We see that if $i = 0$, $\left \lfloor \frac{i+2}{2} \right \rfloor > \left \lfloor \frac{i+1}{2} \right \rfloor + \left \lfloor \frac{i}{6} \right \rfloor$. Also, since $0 \leq i \leq 6n-1$, $ \left \lfloor \frac{i+1}{2} \right \rfloor + \left \lfloor \frac{i}{6} \right \rfloor - \left \lfloor \frac{i+2}{2} \right \rfloor \leq n-1$.\\
    
    \par We will prove the first case, where $0 \leq 3a \leq n-1$. 
    
    \par We want to find the largest $i$ such that $\left \lfloor \frac{i+1}{2} \right \rfloor + \left \lfloor \frac{i}{6} \right \rfloor = 3a-1$, which occurs when $i$ is even. Note that when $i$ is even, $\left \lfloor \frac{i+1}{2} \right \rfloor + \left \lfloor \frac{i}{6} \right \rfloor = \frac{i}{2} + \left \lceil \frac{i+1}{6} \right \rceil-1$, so we have $i$ even and $\frac{i}{2} + \left \lceil \frac{i+1}{6} \right \rceil = 3a$, which appears in the rightmost bound of $I(i)$ when $i$ is even. Therefore, if $j \equiv a\ell$ modulo $n$ and 
    \[0 < j \leq \left \lfloor \frac{3a\ell}{4}\right \rfloor,\]
    then $f(j) = 1$. Now we solve for the smallest $i$ such that $\left \lfloor \frac{i+2}{2}\right \rfloor = 3a+1$, which is $i = 6a$. Then, find the largest $i$ such that $\left \lfloor \frac{i+1}{2} \right \rfloor + \left \lfloor \frac{i}{6} \right \rfloor = n+3a-1$. By continuing this process, we find that if $j \equiv a\ell$ modulo $n$, $f(j)=1$ if and only if 

    \[j \in \left(0, \frac{3a\ell}{4}\right) \cup \left(a\ell, \frac{\ell(n+3a)}{4} \right) \cup \left(\frac{\ell(n+3a)}{3}, \frac{\ell(2n+3a)}{4} \right) \cup \left(\frac{\ell(2n+3a)}{3}, \frac{\ell(3n+3a)}{4} \right).\]

\end{proof}

Now we can prove Theorem \ref{slopes}. By using Proposition \ref{orbits}, the lemmas in Section \ref{qe}, and the Shimura-Taniyama formula, we can determine the slopes of the Newton polygon of the cover described in Construction \ref{construction}.

\begin{proof}[Proof of Theorem \ref{slopes}]
Let $f$ be the signature type of the cover in Construction \ref{construction}. Let $j \in \{1,2,\ldots,\ell n-1\}$ such that $j \equiv a\ell$ modulo $n$ for some $0 \leq a \leq n-1$ and $j \not \equiv 0$ modulo $\ell$.

First, assume $0 \leq 3a \leq n-1$, then Lemma \ref{intervals} says that $f(j) =1$ if and only if 

\begin{align}
0 < j& < \frac{3a\ell}{4} \label{first_i} \\
a\ell < j& < \frac{\ell(n+3a)}{4} \label{second_i} \\
 \frac{(n+3a)\ell}{3}  < j& < \frac{(2n+3a)\ell}{4} \label{third_i} \\
 \frac{(2n+3a)\ell}{3}  <  j& <  \frac{(3n+3a)\ell}{4} \label{fourth_i}
\end{align}

We wish to compute the number of quadratic residues and nonresidues inside of 
\begin{align} \left(0, \frac{3a\ell}{4}\right) \cup \left(a\ell, \frac{\ell(n+3a)}{4} \right) \cup \left(\frac{\ell(n+3a)}{3}, \frac{\ell(2n+3a)}{4} \right) \cup \left(\frac{\ell(2n+3a)}{3}, \frac{\ell(3n+3a)}{4} \right)\label{first-a-interval} \end{align}

and show that it has the same number of quadratic residues and nonresidues as 
\begin{align}\left(0, \frac{\ell}{4} \right) \cup \left(\frac{\ell}{3}, \frac{\ell}{2} \right) \cup \left(\frac{2\ell}{3}, \frac{3\ell}{4} \right). \label{needed_interval} \end{align}

We note that intervals $(\ref{first-a-interval})$ and $(\ref{needed_interval})$ contain $g$ integers, which is the sum of the number of non-zero quadratic residues and nonresidues in these intervals. The quadratic excess of any of these intervals is the number of quadratic residues minus the number of quadratic nonresidues in the interval. Therefore, if the quadratic excess of these intervals is equal, then they have the same number of quadratic residues and the same number of quadratic nonresidues. The number of non-zero quadratic residues in this interval gives one slope of the Newton polygon, and the number of quadratic nonresidues gives the other slope.

As $j \equiv a\ell$ modulo $n$, we can add $(n-1)a\ell$ to interval (\ref{first_i}) to get elements congruent to $0$ modulo $n$. As we are adding a multiple of $\ell$, this doesn't change the number of quadratic residues or nonresidues in the interval. For intervals (\ref{second_i}), (\ref{third_i}), and (\ref{fourth_i}), subtract $a\ell$. Then, divide each inequality by $n$. Finally, we shift the resulting interval from interval 
(\ref{first_i}) by subtracting $(a-1)\ell$.  We now need to compute

\begin{align}
    q\left(\frac{3\ell(n-a)}{3n}, \frac{\ell(4n-a)}{4n}\right) + q\left(0, \frac{\ell(n-a)}{4n} \right) + q\left(\frac{n\ell}{3n}, \frac{\ell(2n-a)}{4n} \right) + q\left(\frac{2n\ell}{3n}, \frac{\ell(3n-a)}{4n} \right). \label{first-qe}
\end{align}

We see that

\begin{align}
    q\left(\frac{3\ell(n-a)}{3n}, \frac{\ell(4n-a)}{4n}\right) = \left (\frac{-1}{\ell} \right) q \left(\frac{a\ell}{4n}, \frac{3a\ell}{3n} \right) = \left (\frac{-1}{\ell} \right) q \left(\frac{a\ell}{4n}, \frac{4a\ell}{4n} \right) \label{interval-manip}
\end{align}

Finally, by substituting equation (\ref{interval-manip}) into Lemma \ref{lemma} with $r = a$, we get that equation (\ref{first-qe}) equals

\begin{align*}
    q\left(0, \frac{\ell}{4} \right) + q\left(\frac{\ell}{3}, \frac{\ell}{2} \right) + q\left(\frac{2\ell}{3}, \frac{3\ell}{4} \right).
\end{align*}

This concludes the case when $0 \leq 3a \leq n-1$.

Assume $n \leq 3a \leq 2n-1$, then Lemma \ref{intervals} states that if $j \equiv a\ell$ modulo $n$ and $j \not \equiv 0 $ modulo $\ell$, then $f(j) = 1$ if and only if $j$ lies in one of the following intervals:

\begin{align}
    0 < j& < \frac{\ell(3a-n)}{4} \label{2a-1}\\
    \frac{\ell(3a-n)}{3} < j& < \frac{3\ell a}{4} \label{2a-2}\\
    a\ell < j& < \frac{\ell(n+3a)}{4} \label{2a-3}\\
    \frac{\ell(n+3a)}{3} < j& < \frac{\ell(2n+3a)}{4} \label{2a-4}
\end{align}

Note that intervals (\ref{2a-3}) and (\ref{2a-4}) are the same as intervals (\ref{second_i}) and (\ref{third_i}), respectively, so we only need to prove that the quadratic excess of the intervals (\ref{2a-1}) and (\ref{2a-2}) equals the quadratic excess of intervals (\ref{first_i}) and (\ref{fourth_i}). We add $a\ell(n-1)$ to both (\ref{2a-1}) and (\ref{2a-2}), divide by $n$, and then subtract $\ell(a-1)$. We now count the quadratic excess

\begin{align*}
   q&\left(\frac{3\ell(n-a)}{3n}, \frac{\ell(3n-a)}{4n}\right) + q\left(\frac{2n\ell}{3n}, \frac{\ell(4n-a)}{4n}\right)  \\ =  q&\left(\frac{3\ell(n-a)}{3n}, \frac{\ell(4n-a)}{4n} \right) - q\left(\frac{\ell(3n-a)}{4n} \frac{\ell(4n-a)}{4n} \right) \\ + \ q&\left(\frac{2n\ell}{3n}, \frac{\ell(3n-a)}{4n} \right) + q\left(\frac{\ell(3n-a)}{4n},\frac{\ell(4n-a)}{4n} \right)
\end{align*}

Therefore, we get the same quadratic excess as in equation (\ref{first-qe}).

For the final case, assume that $2n \leq 3a \leq 3n-1$. Lemma \ref{intervals} state that when $j \equiv a\ell$ modulo $n$, $j \not \equiv 0$ modulo $\ell$, and $f(j) = 1$:

\begin{align}
    0 < j& < \frac{\ell(3a-2n)}{4} \label{3a-1} \\
    \frac{\ell(3a-2n)}{3} < j& < \frac{\ell(3a-n)}{4} \label{3a-2} \\
    \frac{\ell(3a-n)}{3} < j & < \frac{3\ell a}{4} \label{3a-3} \\
    a\ell < j& < \frac{\ell(n+3a)}{4} \label{3a-4}
\end{align}

As intervals (\ref{3a-3}) and (\ref{3a-4}) match intervals (\ref{2a-2}) and (\ref{2a-3}), respectively, we only have to focus on intervals (\ref{3a-1}) and (\ref{3a-2}). We get the quadratic excess

\begin{align*}
    q\left(\frac{3\ell(n-a)}{3n}, \frac{\ell(2n-a)}{4n} \right) + q\left(\frac{n\ell}{3n}, \frac{\ell(3n-a)}{4n} \right)
\end{align*}

and this equals the quadratic excess coming from intervals (\ref{2a-1}) and (\ref{2a-4}).
\end{proof}

\par Corollary \ref{conjecture-corollary} follows immediately from the Theorem. We now prove the remaining corollaries: 

\begin{proof}[Proof of Corollary \ref{germain1}]
    Since $\ell$ is $1$ modulo $4$, $\left( \frac{-1}{\ell} \right) = 1$. This gives us that $q\left( \frac{2\ell}{3}, \frac{3\ell}{4} \right) = q\left(\frac{\ell}{4}, \frac{\ell}{3} \right)$. We then see that 
    
    \[q\left(0, \frac{\ell}{4} \right) + q\left(\frac{\ell}{3}, \frac{\ell}{2} \right) + q\left(\frac{2\ell}{3}, \frac{3\ell}{4} \right) = q\left(0,\frac{\ell}{2}\right) = 0.\]
    
    This means the number of quadratic residues equals the number of quadratic nonresidues in $(0,\ell/4) \cup (\ell/3,\ell/2) \cup (2\ell/3,3\ell/4)$, so the only slope of the Newton polygon is $1/2$.
\end{proof}

\begin{proof}[Proof of Corollary \ref{germain2}]
    Follow the proof of Proposition \ref{orbits} with $p$ having order $2g$ in the group of units modulo $\ell$. Note that this means that $p$ is a quadratic nonresidue modulo $\ell$ and $p \not \equiv -1$ modulo $\ell$. In proposition \ref{orbits}, each orbit relevant to the Shimura-Taniyama method is either composed entirely of quadratic residues or nonresidues, and they contain all quadratic residues or nonresidues in a congruence class modulo $n$. As $p$ is a nonresidue, if we multiply $p$ by a quadratic residue, we get a quadratic nonresidue. Similarly, if we multiply $p$ by a quadratic nonresidue, we get a quadratic residue. Therefore, each orbit is the union of the orbits corresponding to the congruence classes modulo $n$.
    
    For the orbits in Proposition \ref{orbits} that contain a congruence class modulo $n$, there are $mg$ elements for some multiple $m$. Also, there are $m\alpha$ elements with signature type $1$ in one of the orbits, and $m(g-\alpha)$ elements with signature type $1$ in the other orbit. Therefore, when $p$ has multiplicative order $2g$ modulo $\ell$, the orbit corresponding to a congruence class modulo $n$ has $2gm$ elements in it and $mg$ elements with signature type one, so by the Shimura-Taniyama formula we get the slope $1/2$.
\end{proof}

\section{Large Denominators and Density}

\subsection{Large Denominators}
\par \hspace{\parindent} If $g$ is an odd Sophie germain prime and $\ell = 2g+1$, then Theorem \ref{slopes} shows there are infinitely many covers defined over $\overline{\mathbf{F}_p}$ with a Newton polygon whose slopes all have denominator $g$. For $g \geq 11$, this gives us examples of curves whose Newton polygons have large denominators as described in \cite[Expectation 8.5.4]{oort_denominators}. Proposition \ref{denominators} states when these Newton polygons are unlikely.

\begin{proof}[Proof of Proposition \ref{denominators}]
    If $p$ is a prime that satisfies the conditions in Theorem \ref{slopes} and $\ell = 2g+1$, then there is a curve of genus $ng$ defined over $\overline{\mathbf{F}_p}$ such that if $\alpha$ is the number of quadratic residues and in the intervals $(0,\ell/4)$, $(\ell/3,\ell/2)$, and $(2\ell/3,3\ell/4)$, the Newton polygon of the curve has slopes $\alpha/g$ and $(g-\alpha)/g$. Let $a = \min{\{\alpha,g-\alpha\}}$ and $b = g-a$, so $a/g < 1/2$. We see that if we multiply every element in the interval $(2\ell/3,3\ell/4)$ by $-1$ modulo $\ell$, all of those elements are in the interval $(\ell/4,\ell/3)$. Therefore there are $g$ elements total contained in these intervals.
    
    Now consider the quadratic excess in each of the three intervals above. We know that $\ell$ is either $11$ or $23$ modulo $24$, and by the tables in \cite[Section 3]{johnson_symmetries_1977}, in each of those cases the quadratic excess in at least one of those intervals is $0$. By the Pólya-Vinogradov inequality \cite[Chapter 23]{multnumtheory}, we have that \[-2\sqrt{2g+1}\log{(2g+1)} < a-b < 2\sqrt{2g+1}\log{(2g+1)}.\] Since $a+b = g$, we find that \[\frac{g-2\sqrt{2g+1}\log{(2g+1)}}{2g} < \frac{a}{g} < \frac{1}{2}.\] Since $f(g) := (g - 2\sqrt{2g+1}\log{(2g+1)})/(2g)$ has positive second derivative for $g > 0$ and converges to $1/2$ as $g$ approaches infinity, by doing a quick computation we have that if $g > 341$, then $f(g) > 0$ and $f(g)$ increases as $g$ increases.
    
    Let $\xi_1$ be the symmetric Newton polygon of height $2g$ with two slopes, where the smaller slope is $\lfloor gf(g) \rfloor / g < a/g$. Let $\xi_2$ be the symmetric Newton polygon of height $2g$ with slopes $a/g$ and $b/g$. Then, dim$(\mathcal{A}_g[\xi_1]) > $ dim$(\mathcal{A}_g[\xi_2])$ and the codimension of $\mathcal{A}_g[\xi_2]$ is bigger than the codimension of $\mathcal{A}_g[\xi_2]$ in $\mathcal{A}_g$. We use Theorem \ref{npdim} to compute the dimension of $\mathcal{A}_g[\xi_1]$, and we show $\xi_1$ is unlikely and therefore $\xi_2$ is unlikely. By applying Pick's Theorem we can show that 
    
    \[\textrm{dim}(\mathcal{A}_g[\xi_1]) = \frac{g^2-g\lfloor gf(g) \rfloor -\lfloor gf(g) \rfloor+1}{2}\] and the codimension of $\mathcal{A}_g[\xi_1]$ in $\mathcal{A}_g$ is \[\frac{g+g\lfloor gf(g) \rfloor+\lfloor gf(g) \rfloor-1}{2}.\]
    
    The codimension is increasing as $g$ increases if $g > 341$, and a computation shows that the codimension exceeds $3g-3$ if $g > 363$, and therefore if $g > 363$ is a Sophie Germain prime, $\xi_1$ is unlikely and has large denominators.
    
\end{proof}

\begin{example}
    Let $g = 419$, which is a Sophie Germain prime. If $n$ is a power of $2$, then Proposition \ref{denominators} tells us that there exists a curve of genus $ng$ defined over $\overline{\mathbf{F}_3}$ whose Newton polygon has slopes $193/419$ and $226/419$ and is unlikely.
\end{example}

\begin{remark}
If Dickson's conjecture is true, then if we consider the density of primes such that there exists a curve of genus $g$ that has an unlikely Newton polygon with slopes that have large denominators, then Proposition \ref{denominators} shows that the limit superior as $g$ goes to infinity is $1$. Even without Dickson's conjecture, there are very large Sophie Germain primes known. For example, for $g = 183027 \times 2^{265440} - 1$ and any prime $p \not \equiv -1$ modulo $(2g+1)$, there exists a curve of arbitrarily large genus defined over $\overline{\mathbf{F}_p}$ with an unlikely Newton polygon which has slopes that all have denominator $g$.
\end{remark}

\subsection{Results on Density}
\par \hspace{\parindent} Let $g_1, g_2, \ldots, g_k$ be distinct odd Sophie Germain primes such that for $1 \leq i,j \leq k$, $g_i \neq 2g_j+1$ and $g_i \nmid (g_j-1)$. Let $a \geq 0$ and for $1 \leq i \leq k$ define $n_i := 2^a \prod_{b \neq i} g_b$. Then, $n_i$ is coprime to $2g_i+1$, $g_in_i = g_jn_j$, and for $i \neq j$, gcd$(g_i, \varphi(n_j)) = 1$ where $\varphi$ is Euler's totient function. Therefore, if $p$ is any prime with order $2g_i$ in $\mathbf{Z}/(2g_i+1)\mathbf{Z}$ for any $i$, then there is a supersingular curve of genus $n_ig_i$ defined over $\overline{\mathbf{F}_p}$ by Corollary \ref{germain2}. We can use this construction to give a lower bound for the limit superior of $\delta_{ss}(g)$.
\par We know that $p$ having order $2g_i$ is equivalent to $\ell$ being a quadratic nonresidue modulo $2g_i+1$ such that $p \not \equiv -1$ modulo $2g_i+1$. There are exactly $g_i -1$ congruence classes that satisfy this, so the prime density in the case of one Sophie Germain prime is $(g_i-1)/(2g_i)$. We use the Inclusion-Exclusion Principle to compute the density of such primes given $k$ Sophie Germain primes as above:
\begin{align}\sum_{i=1}^k \frac{g_i-1}{2g_i} - \sum_{1 \leq i < j \leq k}\frac{(g_i-1)(g_j-1)}{4g_ig_j} + \ldots + (-1)^{k+1} \frac{(g_1-1)(g_2-1)\ldots(g_k-1)}{2^kg_1g_2\ldots g_k} \label{ie} \end{align}

\begin{proof}[Proof of Proposition \ref{limsup}]
    Consider the following set of Sophie Germain primes: 
    
    \begin{align*} S = \{&3, 5, 29, 41, 53, 83, 89, 113, 131, 173, 179, 191, 233, 239, 251, 281, 293, 359, 419, 431, 443, 491,\\ &509, 593, 641, 653, 659, 683, 719, 743, 761, 809, 911, 953, 1013, 1019, 1031, 1049, 1103, 1223,\\ &1229, 1289, 1409, 1439, 1451, 1481, 1499, 1511, 1559\}. \end{align*}
    
    If $g_1,g_2$ are in this set then $g_1 \neq 2g_2+1$ , $g_2 \neq 2g_1+1$, $g_1 \nmid (g_2-1)$, and $g_2 \nmid (g_1-1)$. If $g = \prod_{s \in S} s$, then by (\ref{ie}) the set of primes such that for all $a \geq 0$ there is a supersingular curve of genus $2^ag$ is greater than $0.9999$.
\end{proof}

We note that if we let $g_1,g_2,\ldots,g_k$ in (\ref{ie}) go to infinity, then (\ref{ie}) converges to $(2^k-1)/2^k$. Therefore, if we have arbitrarily many Sophie Germain primes that satisfy the conditions above, we can show that $\limsup_{g \to \infty} \delta_{ss}(g) = 1$.

\begin{definition}
    Let $k$ be a positive integer. A \textit{Cunningham chain of the first kind} of length $m$ is a sequence of prime numbers $\ell_1,\ell_2,\ldots,\ell_m$ such that if $2 \leq i \leq m$, $\ell_i = 2\ell_{i-1}+1$ \cite[Section 3.6]{caldwell}. In particular, $\ell_1,\ell_2,\ldots,p_{m-1}$ are Sophie Germain primes. 
\end{definition}
Let $m \geq 1$ and $\ell > 2^{2m-2}-1$ be a Sophie Germain prime. Assume there exists a Cunningham chain of the first kind of length $2m$ with $\ell_1 = \ell$. Then, if we choose the Sophie Germain primes $\ell_1, \ell_3, \ell_5, \ldots, \ell_{2m-1}$ and relabel them as $g_1,g_2, \ldots, g_m$, then for $1 \leq i,j \leq m$, $g_i \neq 2g_j+1$. Also, as $g_m = 2^{2m-2}g_1+2^{2m-2}-1$ and $\ell > 2^{2m-2}-1$ we have that $g_i \nmid (g_j-1)$.

\begin{conjecture}[Dickson] \label{conjecture:dickson}
    \cite[Conjecture 1.3]{caldwell} If $m \geq 1$, $a_1,a_2,\ldots,a_m$ are integers, $b_1,b_2,\ldots,b_m$ are all integers greater than or equal to $1$, and there is no prime that divides all of $a_1+b_1n, a_2+b_2n,\ldots, a_m+b_mn$ for every $n$, then there are infinitely many $n$ such that all of $a_1+b_1n, a_2+b_2n,\ldots, a_m+b_mn$ are prime.
\end{conjecture}

\begin{proof}[Proof of Theorem \ref{dickson}]
    Let $a_i = 2^{i-1}-1$ and $b_i = 2^{i-1}$ for $1 \leq i \leq 2m$, then $a_1+b_1n = n$ and $a_2+b_2n = 2n+1$. If $p$ divides $n$, then $p$ does not divide $2n+1$. Therefore, Dickson's conjecture implies there are infinitely many Cunningham chains of the first kind of length $2m$.

    \par Since there are infinitely many chains of length $2k$, we can always choose one where the first prime is bigger than $2^{2k-2}-1$. Therefore, if Dickson's conjecture is true, $\limsup_{g \to \infty} \delta_{ss}(g) = 1$. 
\end{proof}

\begin{remark}
    As supersingular curves have an unlikely Newton polygon when $g \geq 9$, Proposition 4.6 and Theorem 4.8 also show that $\limsup_{g \to \infty} \delta_{nu}(g) > 0.999$, and if Dickson's Conjecture is true, $\limsup_{g \to \infty} \delta_{nu}(g) = 1$.
\end{remark}

\subsection{Proof of a Known Result}
As a final remark, we show that Theorem \ref{slopes} gives another proof of a known but significant result.

\begin{theorem}[\protect{\cite[Theorem 13.7]{vdg}}, \protect{\cite[Corollary 3.7(ii)]{blache}}, \protect{\cite[Proposition 1.8.5]{bhm}}]
    Let $p$ be an odd prime, then for arbitrarily large $g$ there exists a supersingular curve of genus $g$ defined over $\overline{\mathbf{F}_p}$.
\end{theorem}

\begin{proof}
    We need to choose a prime $\ell$ such that $p \not \equiv 1$ modulo $\ell$ and $p$ is a quadratic residue modulo $\ell$. Suppose $\ell \equiv 1$ modulo $4$, then by quadratic reciprocity 
    
    \[\Biggl (\frac{\ell}{p} \Biggr) \Biggl(\frac{p}{\ell} \Biggr) = 1.\]
    
    This shows that $\ell$ must be a quadratic residue modulo $p$, so choose $\ell$ such that $\ell \equiv 1$ modulo $4$ and $\ell \equiv 1$ modulo $p$. As long as we select $\ell$ so that $\ell$ does not divide $p - 1$, $p \not \equiv 1$ modulo $\ell$. Therefore, by Corollary \ref{germain1}, there exists a supersingular curve of genus $\frac{\ell-1}{2}$ defined over $\overline{\mathbf{F}_p}$.
\end{proof}

\bibliography{citations}
\bibliographystyle{alpha}
\end{document}